\documentclass[3p,times]{elsarticle}

\usepackage{amsthm}
\usepackage{amsmath}
\usepackage{bm}
\usepackage{bbm}
\usepackage{float}
\usepackage{hyperref}
\usepackage{cleveref}
\usepackage{amsmath, amsfonts, amssymb}
\usepackage{verbatim}
\usepackage{tikz}
\usepackage{nccmath}
\usepackage{arydshln}
\usetikzlibrary{decorations.pathreplacing,calligraphy}
\usepackage{graphicx}
\usepackage[ruled]{algorithm2e}
\usepackage{graphicx}
\usepackage{caption}
\usepackage{subcaption}
\usepackage[T1]{fontenc}

\newtheorem{lemma}{Lemma}

\begin{document}

	\title{Efficient parallel inversion of ParaOpt preconditioners}
	\author[1]{Corentin Bonte}
	\author[1]{Arne Bouillon\corref{cor1}}\ead{arne.bouillon@kuleuven.be}
	\author[1]{Giovanni Samaey}
	\author[1]{Karl Meerbergen}

	\cortext[cor1]{Corresponding author}
	\affiliation[1]{organization={Department of Computer Science, KU Leuven},
		addressline={Celestijnenlaan 200a -- bus 2402},
		postcode={3001},
		city={Leuven},
		country={Belgium}}

\begin{abstract}
    Recently, the ParaOpt algorithm was proposed as an extension of the time-parallel Parareal method to optimal control. ParaOpt uses quasi-Newton steps that each require solving a system of matching conditions iteratively. The state-of-the-art parallel preconditioner for linear problems leads to a set of independent smaller systems that are currently hard to solve. We generalize the preconditioner to the nonlinear case and propose a new, fast inversion method for these smaller systems, avoiding disadvantages of the current options with adjusted boundary conditions in the subproblems.
\end{abstract}
\maketitle
\section{Introduction} \label{sec:intro}
    Due to the rise of parallel computers, in recent years a significant amount of work has gone into designing and analyzing parallel-in-time integration methods. We improve upon one of these methods, ParaOpt. It was proposed in \cite{ganderPARAOPTPararealAlgorithm2020c} as a parallel-in-time algorithm to solve the optimal-control problem of minimizing a \emph{final-value} objective function:
\begin{align}\label{optimalcontrolproblem2}
	\min_u \frac{1}{2}\| y(T) - y_{\mathrm{target}} \|^2 + \frac{\gamma}{2}\medint\int_0^T \|u(t)\|^2\mathrm dt.
\end{align}
Here, $u(t) \in \mathbb{R}^N$ is a control term and $y(t) \in \mathbb{R}^N$ is a state variable that evolves over $t \in [0, T]$ according to the nonlinear equation $y'(t) = g(y(t)) + u(t)$, with initial condition $y(0) = y_{\text{init}}$. The desired state at the final time $T$ is $y_{\mathrm{target}}$, while $\gamma$ is a regularization parameter. In \cite{bouillonDiagonalizationBasedPreconditionersGeneralized2024}, ParaOpt was adapted to optimize \emph{tracking} objectives:
\begin{align}\label{optimalcontrolproblem}
	\min_u \frac12\medint\int_0^T\|y(t) - y_\mathrm d(t)\|^2\mathrm dt + \frac\gamma2\medint\int_0^T\|u(t)\|^2\mathrm dt.
\end{align}
To simplify exposition, the following sections describe our work in the context of the final-value problem (\ref{optimalcontrolproblem2}). However, it can equally be applied to ParaOpt's tracking variant; numerical tests in \cref{sec:numExample} illustrate this.

The problem (\ref{optimalcontrolproblem2}) can be solved with the method of Langrage multipliers. First-order order optimality conditions then lead to a boundary value problem (BVP), derived in e.g. \cite{ganderConstrainedOptimizationLagrangian2014}, given by
\begin{subequations}\label{BVP2}
	\begin{align}
		y'(t) &= g(y(t)) - \lambda(t)/\gamma,  \hfill &y(0) &= y_{\mathrm{init}}, \\
		\lambda'(t) &= -g'(y(t))^*\lambda(t),  &\lambda(T) &= y(T) - y_{\mathrm{target}}.
	\end{align}	
\end{subequations}
The state variable $y$ and adjoint variable $\lambda$ must satisfy boundary conditions (BCs) at both ends of the time interval. 

Based on the popular Parareal algorithm \cite{lionsResolutionDEDPPar2001} for initial value problems (IVPs), ParaOpt is a multiple-shooting method that solves (\ref{BVP2}) by dividing the time interval into smaller subintervals. An inexact Newton iteration ensures continuity of the solution across subintervals. In each Newton step, the Jacobian is inverted iteratively, for which a preconditioner was proposed in \cite{bouillonDiagonalizationBasedPreconditionersGeneralized2024} in the linear case. Based on diagonalization, its inversion leads to a set of smaller, independent systems that can be solved in parallel. However, solving these systems currently requires either an extra inner iteration or a compromise on the black-box property of ParaOpt's BVP solvers. We propose a new, direct inversion method that avoids these disadvantages by using  adjusted boundary conditions in the subintervals.

We review ParaOpt in \cref{sec:ParaOptIntro}. The current preconditioner for linear equations is discussed in \cref{sec:LinearPrecon}, which we generalize to the nonlinear case in \cref{sec:NonLinearPrecon}. \Cref{sec:NewMethod} introduces a new inversion method for the inner systems. Finally, \cref{sec:numExample} illustrates the improvements made with the new method by means of a short numerical example.

\section{The ParaOpt algorithm} \label{sec:ParaOptIntro}
	In this section, we review the basic ParaOpt algorithm as described in \cite{ganderPARAOPTPararealAlgorithm2020c}. The method first divides the time domain $[0, T]$ into $L$ subintervals $[T_{\ell-1}, T_{\ell}]$ for $\ell=1,\hdots,L$, with $T_{\ell} - T_{\ell-1} = \Delta T$. The values of the state $y$ and the adjoint $\lambda$ on a grid point $T_\ell$ are denoted as $Y_\ell$ and $\Lambda_\ell$, respectively. Finally, two propagators $\mathcal{P}(Y_{\ell-1}, \Lambda_{\ell}, T_{\ell-1}, T_{\ell})$ and $\mathcal{Q}(Y_{\ell-1}, \Lambda_{\ell}, T_{\ell-1}, T_{\ell})$ are introduced. Together, they solve the BVP (\ref{BVP2}) locally in the subinterval $[T_{\ell-1}, T_{\ell}]$ with adapted boundary conditions $y(T_{\ell-1}) = Y_{\ell-1}$ and $ \lambda(T_\ell) = \Lambda_\ell$. The operator $\mathcal{P}$ approximates $y(T_\ell)$, while $\mathcal Q$ approximates $\lambda(T_{\ell-1})$. To keep notation light, the last two arguments of $\mathcal P$ and $\mathcal Q$ will be left out, as they are clear from context. 

To ensure that the concatenated solutions of the local subproblem solve (\ref{BVP2}), its BCs need to be met and the solution needs to be continuous across subintervals. This leads to a set of matching conditions that form the nonlinear system
\begin{align}
\label{ParaOptEquations}
f\left(\begin{bmatrix}
	\bm{Y} \\ \bm{\Lambda}
\end{bmatrix}\right) := 
\left[ 
\begin{array}{c}
	Y_1 - \mathcal{P}(y_{\text{init}}, \Lambda_1 ) \\
	Y_2 - \mathcal{P}(Y_1, \Lambda_2 ) \\
	\vdots \\
	Y_{{L}} - \mathcal{P}(Y_{{L}-1}, \Lambda_{{L}}) \\
	\hline
	\Lambda_1 - \mathcal{Q}(Y_1, \Lambda_2) \\
	\vdots \\
	\Lambda_{{L}-1} - \mathcal{Q}(Y_{{L}-1}, \Lambda_{{L}}) \\
	\Lambda_{{L}} - (Y_{{L}} - y_{\text{target}})
\end{array} \right] = 0.
\end{align}
This ParaOpt system is solved with Newton's method and initial guess $(\bm Y^0, \bm\Lambda^0)$. Evaluating the propagators' derivatives, denoted by $\mathcal P_y$, $\mathcal Q_y$, $\mathcal P_\lambda$, and $\mathcal Q_\lambda$, to compute the Jacobian $f'$ is expensive. Hence, the derivatives $\mathcal P^G_y$, $\mathcal Q^G_y$, $\mathcal P^G_\lambda$, and $\mathcal Q^G_\lambda$ of \emph{coarse} versions of $\mathcal{P}$ and $\mathcal{Q}$, denoted by $\mathcal P^G$ and $\mathcal Q^G$, are used to approximate $f'$. This results in the iteration
\begin{align} \label{ParaOptNewtonIteration}
\Tilde f'\left( \begin{bmatrix}
	\bm{Y}^{k-1} \\ \bm{\Lambda}^{k-1}
\end{bmatrix}\right) \begin{bmatrix}
\bm{Y}^{k} - \bm{Y}^{k-1} \\ \bm{\Lambda}^{k} - \bm{\Lambda}^{k-1}
\end{bmatrix} = f\left( \begin{bmatrix}
\bm{Y}^{k-1} \\ \bm{\Lambda}^{k-1}
\end{bmatrix}\right),
\end{align}
where $k$ is the iteration index and $\Tilde f'$ approximates $f'$ by using the coarse propagators. The ParaOpt algorithm then becomes an inexact Newton iteration, in which each step is solved iteratively with a GMRES solver.

To evaluate products with $\Tilde f'$, we need products with the derivatives of $\mathcal P^G$ and $\mathcal Q^G$. ParaOpt does this by coarsely solving parallel subproblems obtained by differentiating the BVP (\ref{BVP2}). For a small perturbation $\delta y$, we approximate $\mathcal{P}^G_y(Y_{\ell-1}^{k-1}, \Lambda_\ell^{k-1})\delta y$ and $\mathcal{Q}^G_y(Y_{\ell-1}^{k-1}, \Lambda_\ell^{k-1})\delta y$ by setting $z(T_{\ell-1}) = \delta y$ and $\mu(T_\ell) = 0$ and solving the linear BVP
\begin{subequations}\label{DerivativeSystem}
\begin{align} 
z'(t) &= g'(y(t))z(t) - \frac{\mu(t)}{\gamma}, \\
\mu'(t) &= - H(y(t), \lambda(t))^*z(t) - g'(y(t))^*\mu(t),
\end{align}
\end{subequations}
where $H(y(t), \lambda(t))^*$ is the Jacobian of $g'(y(t))^*\lambda(t)$. This gives the desired values $z(T_\ell) \approx \mathcal{P}^G_y(Y_{\ell-1}^{k-1}, \Lambda_\ell^{k-1})\delta y$ and $\mu(T_{\ell-1}) \approx \mathcal{Q}^G_y(Y_{\ell-1}^{k-1}, \Lambda_\ell^{k-1})\delta y$. In the same fashion, solving the derivative system with boundary values $z(T_{\ell-1}) = 0$ and $\mu(T_\ell) = \delta \lambda$ results in $z(T_\ell) \approx \mathcal{P}^G_{\lambda}(Y_{\ell-1}^{k-1}, \Lambda_\ell^{k-1})\delta \lambda$ and $\mu(T_{\ell-1}) \approx \mathcal{Q}^G_{\lambda}(Y_{\ell-1}^{k-1}, \Lambda_\ell^{k-1})\delta \lambda$. The variables $y(t)$ and $\mu(t)$ in (\ref{DerivativeSystem}) can be reused from solving (\ref{BVP2}) with $y(T_{\ell-1}) = Y_{\ell-1}$ and $\lambda(T_\ell) = \Lambda_\ell$ while evaluting the right-hand side of (\ref{ParaOptNewtonIteration}).

We note that these two BVPs can actually be combined into one. Since the BVP (\ref{DerivativeSystem}) is linear, solving a single BVP with boundary conditions $z(T_{\ell-1}) = \delta y$ and $\mu(T_\ell) = \delta \lambda$ leads to the summed values $z(T_\ell)\approx\mathcal{P}^G_y\delta y + \mathcal{P}^G_\lambda\delta \lambda$ and $\mu(T_{\ell-1})\approx\mathcal{Q}^G_y\delta y + \mathcal{Q}^G_\lambda\delta \lambda$. Only these sums are needed in a matrix-vector product with the Jacobian.

\section{ParaOpt preconditioners} \label{sec:Preconditioner}
	The basic ParaOpt method is very easily parallelizable, since the two most expensive operations -- computing the residual \eqref{ParaOptEquations} entry by entry or executing a matrix-vector multiplication with $\tilde{f}$ by solving $L$ instances of (\ref{DerivativeSystem}) -- are embarrassingly parallel over the subintervals. However, inverting $\tilde{f}$ iteratively could limit ParaOpt's scalability. If the number of subintervals $L$ and the number of used processors are increased proportionally, computation time could still increase if the number of GMRES iterations required to solve (\ref{ParaOptNewtonIteration}) increases as well. In \cite{bouillonDiagonalizationBasedPreconditionersGeneralized2024}, a parallelizable preconditioner based on diagonalization is introduced for linear ParaOpt. We review this preconditioner in \cref{sec:LinearPrecon} and generalize it to nonlinear equations in \cref{sec:NonLinearPrecon}.

	\subsection{Preconditioning linear ParaOpt}\label{sec:LinearPrecon}
		When the function $g$ is linear, i.e.\ $g(y(t)) = Ky(t)$ for a matrix $K$, typical propagator pairs become affine in their arguments. There then exist matrices 
$P^G_y$, $P^G_\lambda$, $Q^G_y$, and $Q^G_\lambda$, and vectors $b_\mathcal{P}$ and $b_\mathcal{Q}$, such that
\begin{equation} \label{AffineProps}
\mathcal{P}^G(Y_{\ell-1}, \Lambda_\ell) = P^G_y Y_{\ell-1} + P^G_\lambda\Lambda_{\ell} + b_\mathcal{P} \quad \text{and} \quad \mathcal{Q}^G(Y_{\ell-1}, \Lambda_\ell) = Q^G_y Y_{\ell-1} + Q^G_\lambda\Lambda_{\ell} + b_\mathcal{Q}.
\end{equation}
These matrices are independent of $\ell$, such that the Jacobian can be written compactly with the Kronecker product as 
\begin{align}\label{Jacobian_Compact}
\tilde{f}'\left( \begin{bmatrix}
	\bm{Y}^{k-1} \\ \bm{\Lambda}^{k-1}
\end{bmatrix} \right) =  \left[
\begin{array}{cc}
	I \otimes I - B \otimes P^G_y & -I \otimes P^G_\lambda \\
	-I \otimes Q^G_y +  E \otimes (Q^G_y - I) & I \otimes I - B^* \otimes Q^G_\lambda \\
\end{array} \right].	
\end{align}
Here, $B$ is a matrix with ones on its first sub-diagonal and $E$ has a one in the bottom-right corner. Disregarding $E$, this matrix only contains blocks that have a block-Toeplitz structure. Inspired by ParaDiag algorithms \cite{mcdonaldPreconditioningIterativeSolution2018b,ganderParaDiagParalleltimeAlgorithms2021,bouillonGeneralizedPreconditionersTimeParallel2024}, which build preconditioners by replacing block-Toeplitz structures with block-circulant ones, \cite{bouillonDiagonalizationBasedPreconditionersGeneralized2024} proposed a parallel preconditioner $P(\alpha)$ for $\tilde{f}'$ that modifies (\ref{Jacobian_Compact}) in two ways. First, the term $E \otimes (Q^G_y - I)$ is left out, as it destroys the block-Toeplitz structure of the bottom-left block. Second, the Toeplitz matrix $B$ is replaced by the alpha-circulant matrix $C(\alpha)$, which has an extra term $\alpha$ in the top-right corner. It is well-known that alpha-circulants can be diagonalized as
\begin{subequations}
\begin{align}
C(\alpha) = \Gamma_{\alpha}^{-1}\mathbb{F}^* D(\alpha) \mathbb{F}\Gamma_{\alpha},
\end{align}
\end{subequations}
in which $\Gamma_{\alpha}\coloneqq\mathrm{diag}(1, \alpha^{1/L}, \ldots, \alpha^{(L-1)/L})$, $\mathbb{F}$ is the discrete Fourier matrix $\{\mathrm e^{2\pi \mathrm ijk/L}/\sqrt{L}\}_{j,k=0}^{L-1}$, and $D(\alpha)\coloneqq\mathrm{diag}(\sqrt{N}\mathbb{F}c_1)$ with $c_1$ the first column of $C(\alpha)$. If we impose the restriction $|\alpha| = 1$, it holds that $\Gamma_{\alpha}^{-1} = \Gamma_{\alpha}^{*}$ such that $C(\alpha)$ and $C(\alpha)^*$ can be diagonalized simultaneously \cite{bouillonGeneralizedPreconditionersTimeParallel2024}. The following factorization of the preconditioner $P(\alpha)$ is then possible:
\begin{equation}\label{Precon_factorization}
	P(\alpha) = \left(\begin{bmatrix}
		\Gamma^{-1}_{\alpha}\mathbb{F}^*& \\
		& \Gamma^{-1}_{\alpha}\mathbb{F}^* \\
	\end{bmatrix} \otimes I \right)  \left[
\begin{array}{cccc|cccc}
	I \otimes I - D_1(\alpha) \otimes P^G_y & I \otimes -P^G_\lambda \\
	I \otimes - Q^G_y & I \otimes I - D_1^*(\alpha) \otimes Q^G_\lambda
\end{array} \right]
 \left( \begin{bmatrix}
\mathbb{F} \Gamma_{\alpha} & \\
& \mathbb{F}\Gamma_{\alpha}
\end{bmatrix} \otimes I \right).
\end{equation}
The application of $P(\alpha)^{-1}$ to a vector can be decomposed into three steps. The fast Fourier transform can invert the outer two terms in (\ref{Precon_factorization}) with only $\mathcal{O}(L\log L)$ operations. The middle term has four block-diagonal blocks. Inversion can thus be decomposed into $L$ smaller, inner systems that can be solved in parallel. They are of the form
\begin{align} \label{eq:linearinnersystems}
M\begin{bmatrix}
	\delta y \\
	\delta \lambda
\end{bmatrix} \coloneqq \left[
\begin{array}{cccc|cccc}
	I - d_\ell(\alpha) P^G_y & - P^G_\lambda \\
	-Q^G_y & I - d_\ell^*(\alpha) Q^G_\lambda
\end{array} \right] \begin{bmatrix}
	\delta y \\
	\delta \lambda
\end{bmatrix} = \begin{bmatrix}
	\delta p \\
	\delta q
\end{bmatrix},
\end{align}
for some vectors $\delta y$, $\delta\lambda$, $\delta p$, and $\delta q$, with $d_\ell(\alpha)$ the elements of the diagonal matrix $D(\alpha)$. In \cref{sec:NewMethod}, we return to the question of how to solve the inner systems (\ref{eq:linearinnersystems}) efficiently.

	\subsection{Generalization to the nonlinear case}\label{sec:NonLinearPrecon}
		In the nonlinear case, the derivatives of $\mathcal{P}^G$ and $\mathcal{Q}^G$ are not constant anymore, destroying any block-Toeplitz structure in the Jacobian. In nonlinear ParaDiag variants, the preconditioner's desired structure is enforced by using an averaging technique, where nonconstant blocks are replaced by their average \cite{ganderTimeParallelizationNonlinear2017b,liuFastBlock$alpha$Circulant2020,wuUniformSpectralAnalysis2022b}. The same idea can be applied here, except that the derivative blocks (which involve $\mathcal{P}^G_y$, $\mathcal{P}^G_\lambda$, $\mathcal{Q}^G_y$, and $\mathcal{Q}^G_\lambda$) can not be averaged explicitly. Instead, we average $Y_\ell$ and $\Lambda_\ell$ (similarly to, e.g., \cite{ganderParaDiagParalleltimeAlgorithms2021,wuParallelTimeAlgorithmHighOrder2021}); we define two values
\[
Y_\mathrm{av} = \frac{1}{L}\sum_{\ell=1}^{L} Y_\ell \qquad \text{and} \qquad \Lambda_\mathrm{av} = \frac{1}{L}\sum_{\ell=1}^{L} \Lambda_\ell
\]
and use these average state and adjoint in the preconditioner 
\begin{align} \label{ParaOptPrecon}
	P(\alpha) := \left[
	\begin{array}{cccc|cccc}
		I \otimes I - C_1(\alpha) \otimes  \mathcal{P}^G_y(Y_\mathrm{av}, \Lambda_\mathrm{av}) & -I \otimes \mathcal{P}^G_\lambda(Y_\mathrm{av}, \Lambda_\mathrm{av}) \\
		-I \otimes \mathcal{Q}^G_{y}(Y_\mathrm{av}, \Lambda_\mathrm{av}) & I \otimes I - C_1(\alpha)^* \otimes  \mathcal{Q}^G_{\lambda}(Y_\mathrm{av}, \Lambda_\mathrm{av})
	\end{array} \right].
\end{align}
If we set $P^G_y\coloneqq \mathcal P^G_y(Y_\mathrm{av}, \Lambda_\mathrm{av})$, and analogously for $P^G_\lambda$, $Q^G_y$, and $Q^G_\lambda$, inversion of $P(\alpha)$ again results in $L$ smaller systems (\ref{eq:linearinnersystems}). In \cref{sec:NewMethod}, we will discuss how inverting $M$ is done with derivative systems similar to (\ref{DerivativeSystem}). To have $y$ and $\lambda$ available there, we first perform an additional solve of the original system (\ref{BVP2}), with boundary values $y(T_{\ell-1}) = Y_\mathrm{av}$ and $\lambda(T_\ell) = \Lambda_\mathrm{av}$.

\section{Efficiently solving the smaller systems of the preconditioner} \label{sec:NewMethod}
	The paper \cite{bouillonDiagonalizationBasedPreconditionersGeneralized2024} proposes two ways to solve (\ref{eq:linearinnersystems}). A first option introduces a new, possibly expensive, inner GMRES iteration that maintains the black-box property of ParaOpt's propagators, as only products with their derivatives are needed. These can be computed with any solver that can solve the BVP (\ref{DerivativeSystem}). A second option, in some cases, is to use potential known properties of the propagators to avoid the extra inner loop. Propagators can then no longer be black-box. In this section, we propose a new method to invert the inner systems. It maintains the black-box property, is directly applicable to nonlinear equations and avoids introducing another iteration. Thus, it achieves the best of both worlds from the two current inversion methods proposed in \cite{bouillonDiagonalizationBasedPreconditionersGeneralized2024}.

Recall that we want to solve systems with a matrix $M$ of the form (\ref{eq:linearinnersystems}) in both the linear and the nonlinear case. Before treating $M$, we first consider solving systems with simpler matrices $M_0$ and $M_1$. The former simply contains the operators' derivatives:
\begin{align} \label{eq:PO_nonlinearinnersystems_simplified}
	M_0\begin{bmatrix}
		\delta y \\
		\delta \lambda
	\end{bmatrix} \coloneqq \left[
	\begin{array}{cccc|cccc}
		P^G_y & P^G_{\lambda} \\
		Q^G_{y} & Q^G_{\lambda} 
	\end{array} \right] \begin{bmatrix}
		\delta y \\
		\delta \lambda
	\end{bmatrix} = \begin{bmatrix}
		\delta p \\
		\delta q
	\end{bmatrix}.
\end{align}
Instead of solving a BVP to go from $(\delta y, \delta\lambda)$ to $(\delta p, \delta q)$, as we did in \cref{sec:ParaOptIntro}, we now have to go in the other direction. To do this, we change the BCs: we require $z(T_\ell) = \delta p$ and $\mu(T_{\ell-1}) = \delta q$. One then obtains the values $z(T_{\ell-1}) \approx \delta y$ and $\mu(T_{\ell}) \approx \delta \lambda$ by solving just one BVP.

Of course, this is not immediately useful. The matrix $M_0$ in (\ref{eq:PO_nonlinearinnersystems_simplified}) has several simplifications in comparison with the matrix $M$ in (\ref{eq:linearinnersystems}). We first reintroduce the factors $d_\ell(\alpha)$ to get the system
\begin{align}\label{eq:PO_nonlinearinnersystems_simplified_alpha}
		M_1\begin{bmatrix}
			\delta y \\
			\delta \lambda
		\end{bmatrix} \coloneqq
		\left[
		\begin{array}{cccc|cccc}
			d_\ell(\alpha)P^G_y & P^G_{\lambda} \\
			Q^G_{y} & d_\ell^*(\alpha)Q^G_{\lambda} 
		\end{array} \right] \begin{bmatrix}
			\delta y \\
			\delta \lambda
		\end{bmatrix} = \begin{bmatrix}
			\delta p \\
			\delta q
		\end{bmatrix}.
\end{align}
Due to the factors $d_\ell(\alpha)$, this system cannot be solved directly with a BVP with adjusted boundary conditions. However, by the following lemma, $d_\ell(\alpha)$ and $d_\ell^*(\alpha)$ can be factored out.
\begin{lemma}
	Consider the $\alpha$-circulant matrix $C(\alpha) \in \mathbb{C}^{L\times L}$ and its diagonalization $C(\alpha) = \Gamma^{-1}_{\alpha}\mathbb{F}^*D(\alpha)\mathbb{F}\Gamma_{\alpha}$. If the first column of $C(\alpha)$ only has one nonzero element $a$ and if $|\alpha| = 1$, then it follows that the elements $d_\ell(\alpha)$ of the diagonal matrix $D(\alpha)$ have the same magnitude as $a$, i.e.  $|d_\ell(\alpha)| = |a|$ for all $\ell$.
\end{lemma}
\begin{proof}
	The alpha-circulant matrix $C(\alpha)$ is defined by its first column $c_1$. With only one nonzero element $a$, let us denote it as $c_1 = a e_j$ for some unit vector $e_j$. The diagonal matrix $D(\alpha)$ is equal to $\mathrm{diag}(\sqrt{L}\mathbb{F}^{\ast}\Gamma_{\alpha}c_1)$. $\Gamma_{\alpha}$ is diagonal, with the elements $\alpha^{(\ell-1)/L}$ for $\ell=1,\hdots,L$. It follows that $\Gamma_{\alpha}c_1 = a \alpha^{(j-1)/L} e_j$. The matrix $\sqrt{L}\mathbb{F}^{\ast}$, on the other hand, is the rescaled DFT matrix $\{\mathrm e^{2\pi\mathrm ijk}\}_{j,k=0}^{{L}-1}$. It follows that
	\begin{align}
		d_\ell(\alpha) = a \alpha^{(j-1)/L}\mathrm e^{2\pi\mathrm i(j-1)(\ell-1)}.
	\end{align}
	 With the magnitude of $\alpha$ equal to one, the magnitude of $d_\ell$ is $|a|$.
\end{proof}
 In the case of ParaOpt, the relevant $C(\alpha)$ has one nonzero $a=1$ in its first subdiagonal, meaning that $|d_\ell(\alpha)| = 1$. This property allows us to write $d_\ell(\alpha) = \mathrm e^{i\theta}$ for some $\theta$, making the following factorization of $M_1$ possible:
 \begin{align}
	\left[
	\begin{array}{cccc|cccc}
		d_\ell(\alpha)P^G_y & P^G_{\lambda} \\
		Q^G_{y} & d_\ell^*(\alpha)Q^G_{\lambda} 
	\end{array} \right] = \begin{bmatrix}
		e^{i\theta/2}I & \\
		&  e^{-i\theta/2}I
	\end{bmatrix}\left[ \begin{array}{cccc|cccc}
		P^G_y & P^G_{\lambda} \\
		Q^G_{y} & Q^G_{\lambda} 
	\end{array} \right] \begin{bmatrix}
		e^{i\theta/2}I & \\
		& e^{-i\theta/2}I
	\end{bmatrix}.
\end{align} 
The middle term is $M_0$, which we can invert as above, while inversion of the two outer terms is trivial.

Finally, we reintroduce the minus signs and the identity matrices in the diagonal blocks of (\ref{eq:linearinnersystems}). The identities are scaled by the proposed factorization as follows:
\begin{equation}
	M = \left[
	\begin{array}{cc}
		I - d_\ell(\alpha)\mathcal{P}_y & -\mathcal{P}_{\lambda} \\
		-\mathcal{Q}_{y} & I - d_\ell^*(\alpha)\mathcal{Q}_{\lambda} 
	\end{array} \right] = -\begin{bmatrix}
		e^{i\theta/2}I & \\
		&  e^{-i\theta/2}I
	\end{bmatrix}\left[ \begin{array}{cccc|cccc}
		-d_\ell^*(\alpha)I + \mathcal{P}_y & \mathcal{P}_{\lambda} \\
		\mathcal{Q}_{y} & -d_\ell(\alpha)I + \mathcal{Q}_{\lambda} 
	\end{array} \right] \begin{bmatrix}
		e^{i\theta/2}I & \\
		&  e^{-i\theta/2}I
	\end{bmatrix}.
\end{equation}
Inversion of $M$ then requires us to solve a system of the form
\begin{align} \tilde M\begin{bmatrix}
	\delta y \\
	\delta \lambda
\end{bmatrix} \coloneqq \label{eq:PO_nonlinearinnersystems_simplified_alpha_Identity}
	\left[
	\begin{array}{cc}
		-d_\ell^*(\alpha)I + \mathcal{P}_y & \mathcal{P}_{\lambda} \\
		\mathcal{Q}_{y} & -d_\ell(\alpha)I +\mathcal{Q}_{\lambda} 
	\end{array} \right] \begin{bmatrix}
		\delta y \\
		\delta \lambda
	\end{bmatrix} = \begin{bmatrix}
		\delta p \\
		\delta q
	\end{bmatrix}.
\end{align}
Here, the factors $d_\ell^*(\alpha)$ scale the identities instead of propagator derivatives. Consequently, (\ref{eq:PO_nonlinearinnersystems_simplified_alpha_Identity}) modifies the boundary conditions used to invert $M_0$ in (\ref{eq:PO_nonlinearinnersystems_simplified}):
\begin{equation}\label{new_boundary}
	\underbrace{-d_\ell^*(\alpha)\delta y}_{-d_\ell^*(\alpha)z(T_{\ell-1})} + \underbrace{\mathcal{P}_y \delta y + \mathcal{P}_{\lambda} \delta \lambda}_{z(T_\ell)} = \delta p \qquad \text{and} \qquad \underbrace{\mathcal{Q}_y \delta y + \mathcal{Q}_{\lambda} \delta \lambda}_{\mu({T_{\ell-1}})}  \underbrace{- d_\ell(\alpha) \delta \lambda}_{- d_\ell(\alpha)\mu(T_\ell)} = \delta q.
\end{equation}
The system (\ref{eq:PO_nonlinearinnersystems_simplified_alpha_Identity}) thus naturally defines a set of boundary conditions, enabling an efficient inversion of $M$ by solving the corresponding (complex-valued) BVP. The full inversion procedure is given in algorithm \ref{alg:two}. It uses an abstract BVP solver, which maintains the black-box property of ParaOpt's propagators.
\begin{algorithm}
	\caption{Solve the system (\ref{eq:linearinnersystems})}\label{alg:two}
	\SetKwInOut{Input}{input}
	\SetKwInOut{Output}{output}
	
	\Input{Coarse BVP solver, $d_\ell(\alpha)$, $\delta p$, $\delta q$, and the solution $(y(t), \lambda(t))$ of (\ref{BVP2}) in $[T_{\ell-1}, T_\ell]$}
	\Output{Pair $(\delta y, \delta \lambda)$ such that (\ref{eq:linearinnersystems}) holds}
	
	1: Set $\theta\leftarrow\angle d_\ell(\alpha)$ such that $e^{i\theta} = d_\ell(\alpha)$; Set $c \leftarrow e^{i\theta/2}$, $a \leftarrow \delta p / c$, and $b \leftarrow c\delta q$\;
	2: Solve system (\ref{DerivativeSystem}) with boundary conditions $-d_\ell^*(\alpha)z(T_{\ell-1}) + z(T_\ell) = a$ and $\mu({T_{\ell-1}}) - d_\ell(\alpha)\mu(T_\ell) = b$ \; 
	3: Set $\delta y \leftarrow -z(T_{\ell-1}) / c$ and $\delta \lambda \leftarrow -c\mu(T_\ell)$\;
\end{algorithm}

\section{Numerical example} \label{sec:numExample}
	This section briefly illustrates the importance of being able to invert the inner systems efficiently, by means of a nonlinear PDE example. We consider optimal control, for both objectives (\ref{optimalcontrolproblem2}) and (\ref{optimalcontrolproblem}) of the viscous Burgers' equation
\begin{equation}\label{Burgers}
	\frac{\delta y(t, x)}{\delta t} = -\frac{1}{2}\frac{\delta y^2(t, x)}{\delta x} + \nu \frac{\delta^2 y(t, x)}{\delta x^2} + u(t,x),
\end{equation}
with initial value $y(0, x) = f(x)$ and boundary conditions $y(t, 0) = y(t, 1) = 0$. The problem is discretized in space with a pseudo-spectral method \cite{fornbergPracticalGuidePseudospectral1996} with $N=32$ discretization points. We follow \cite{sabehDistributedOptimalControl2016} and consider set-ups with smooth and nonsmooth initial conditions. The first has initial condition $f_1(x) = \sin(4\pi x)$, objectives $y_\mathrm{d}(t) = y_{\mathrm{target} }=0$, and regularization parameter $\gamma = 1$. The nonsmooth set-up has initial condition $f_2(x) = \mathbbm 1_{(0, 1/2]}(x)$, objectives $y_\mathrm{d}(t)= y_{\mathrm{target}} = f_2$, and regularization parameter $\gamma = 0.05$. For both set-ups, $T$ and $\alpha$ are set to $1$ and $\nu$ is set to $0.01$. We use $L=20$ subintervals in ParaOpt, with MATLAB's \texttt{bvp5c} solver as the fine BVP solver and 2-step implicit Euler as the coarse solver. Our code is at \url{https://gitlab.kuleuven.be/numa/public/paper-code-nonlinear-paraopt}.

The total number of \textit{outer} GMRES iterations as the quasi Newton iteration progresses is shown in \cref{fig:numExp:a} for both preconditioned and unpreconditioned tracking ParaOpt. The final value case is shown in \cref{fig:numExp2:a}. We can see that the general, nonlinear preconditioner for $\Tilde f'$ reduces this iteration count. However, this graph hides the true cost of the algorithm. A better measure in this case is the number of coarse, linear BVPs that is solved in one subinterval. When the basic, unpreconditioned ParaOpt algorithm is used, each GMRES iteration requires solving two BVPs in one interval. In contrast, with an extra iteration to invert the preconditioner's inner systems, one GMRES iteration may require many more coarse solves. For this basic example, using a preconditioner actually increases the total number of BVP solves, as shown in \cref{fig:numExp:b,fig:numExp2:b}. When derivatives are computed simultaneously and the newly proposed inversion method is used, the number of coarse solves is reduced (see \cref{fig:numExp:c,fig:numExp2:c}). The reduction matches the reduction in GMRES iterations. For every GMRES iteration, one BVP solve is needed to compute derivatives to $y$ and $\lambda$ simultaneously, and one additional BVP solve is needed in every interval to invert the preconditioner.

\begin{figure}[ht]
	\centering
	\begin{subfigure}[t]{.3\textwidth}
		\centering
		\includegraphics[width=\linewidth,height=3.85cm]{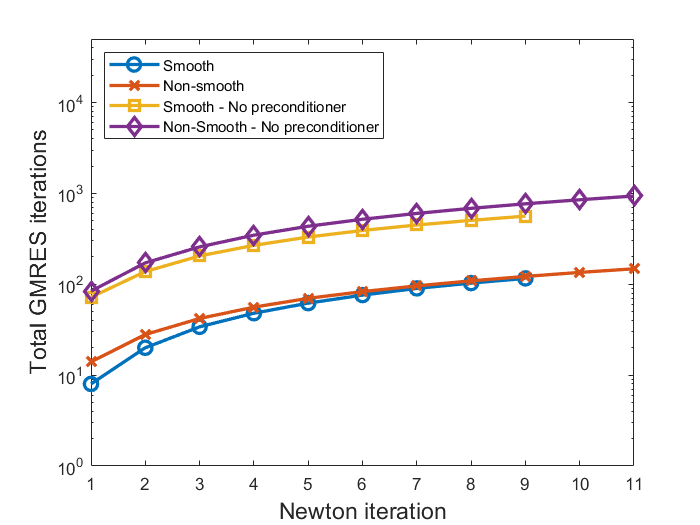}
		\caption{Number of GMRES iterations}
		\label{fig:numExp:a}
	\end{subfigure}%
	\begin{subfigure}[t]{.3\textwidth}
		\centering
		\includegraphics[width=\linewidth,height=3.85cm]{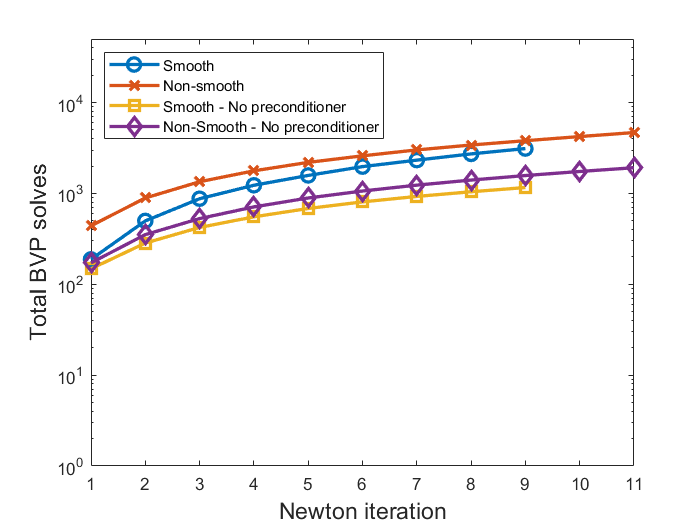}
		\caption{BVP solves with the current method}
		\label{fig:numExp:b}
	\end{subfigure}
	\begin{subfigure}[t]{.3\textwidth}
		\centering
		\includegraphics[width=\linewidth,height=3.85cm]{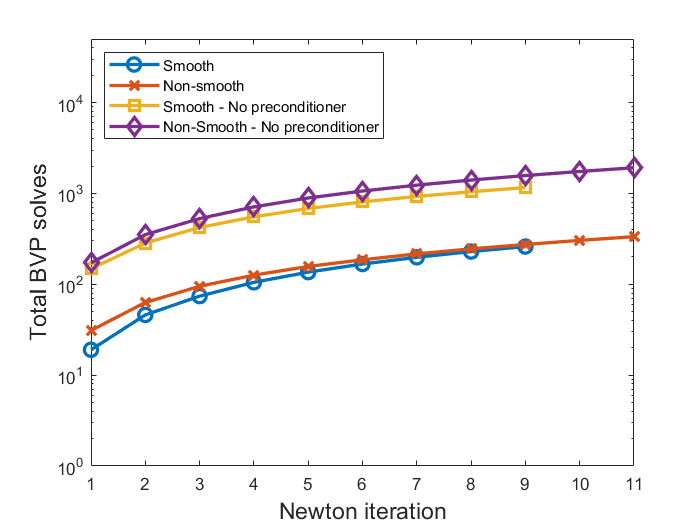}
		\caption{BVP solves with the new, improved method}
		\label{fig:numExp:c}
	\end{subfigure}
	\caption{Total GMRES iterations and coarse BVP solves in one subinterval in tracking ParaOpt (without and with the proposed improvements)}
	\label{fig:numExp}
\end{figure}

\begin{figure}[ht]
	\centering
	\begin{subfigure}[t]{.3\textwidth}
		\centering
		\includegraphics[width=\linewidth,height=3.85cm]{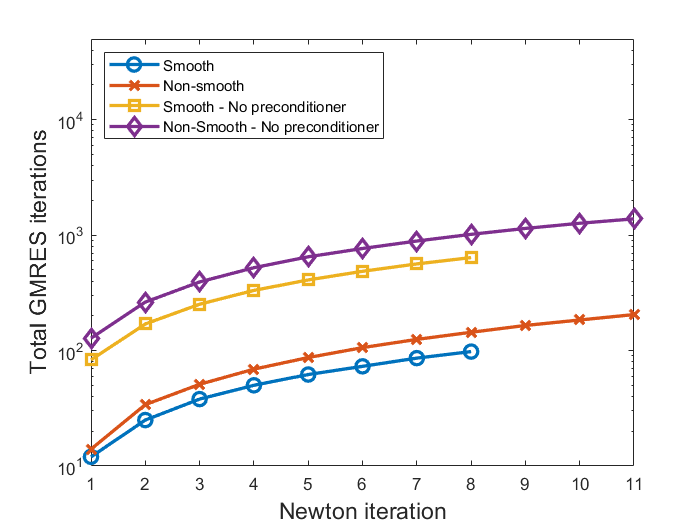}
		\caption{Number of GMRES iterations}
		\label{fig:numExp2:a}
	\end{subfigure}%
	\begin{subfigure}[t]{.3\textwidth}
		\centering
		\includegraphics[width=\linewidth,height=3.85cm]{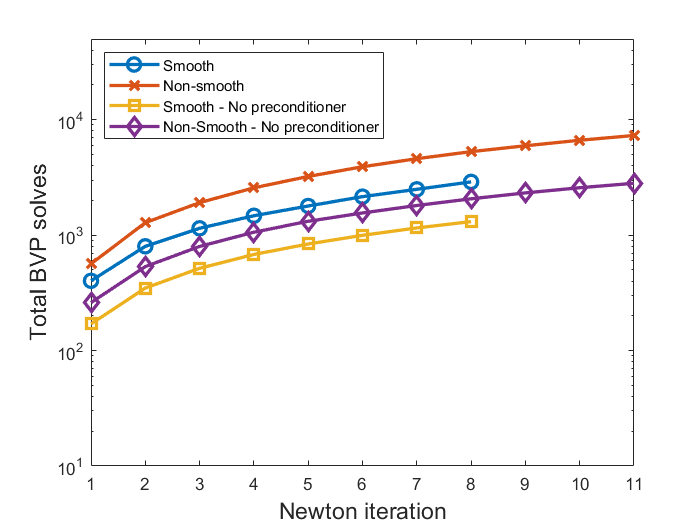}
		\caption{BVP solves with the current method}
		\label{fig:numExp2:b}
	\end{subfigure}
	\begin{subfigure}[t]{.3\textwidth}
		\centering
		\includegraphics[width=\linewidth,height=3.85cm]{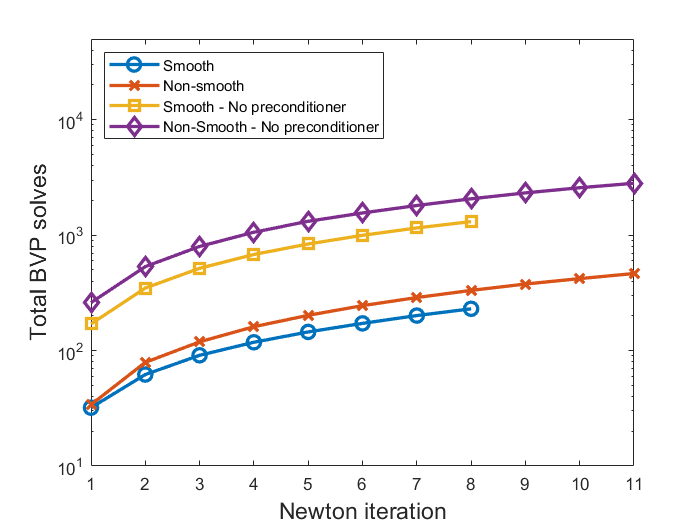}
		\caption{BVP solves with the new, improved method}
		\label{fig:numExp2:c}
	\end{subfigure}
	\caption{Total GMRES iterations and coarse BVP solves in one subinterval in final value ParaOpt (without and with the proposed improvements)}
	\label{fig:numExp2}
\end{figure}

\section*{Acknowledgments}
	We thank Ignace Bossuyt for his helpful comments. Our work was supported in part by the Fonds Wetenschappelĳk Onderzoek -- Vlaanderen (FWO) under grant 1169725N, by Inno4Scale Innovation Study Inno4scale-202301-092, and by the European High-Performance Computing Joint Undertaking (JU) under grant agreement\pagebreak{} No. 955701 (TIME-X). The JU receives support from the European Union's Horizon 2020 research and innovation programme and from Belgium, France, Germany, and Switzerland.

\bibliographystyle{elsarticle-num}
\bibliography{references}

\end{document}